\documentclass[a4paper,12pt]{article}
\usepackage{amsmath}
\usepackage{amsthm}
\usepackage{amsfonts}
\usepackage{amssymb}
\usepackage{latexsym}
\usepackage{mathrsfs}
\usepackage{geometry}
\usepackage{dsfont}
\usepackage{graphicx}
\usepackage{color}
\usepackage[all]{xy}
\usepackage{hyperref}
\usepackage{amsfonts,mathrsfs,amssymb,amsthm}
\usepackage{mathtools}

\usepackage{amsmath,amscd}

\usepackage{color}
\usepackage{hyperref}
\usepackage[all]{xy}
\usepackage{pdfpages}
\usepackage{tikz-cd}
\usepackage{graphicx}
\usepackage{fancyhdr}
\usetikzlibrary{calc}      

\allowdisplaybreaks[4]
\newtheorem{theorem}{Theorem}[section]
\newtheorem{proposition}[theorem]{Proposition}
\newtheorem{lemma}[theorem]{Lemma}
\newtheorem{corollary}[theorem]{Corollary}
\newtheorem{example}[theorem]{Example}
\newtheorem{definition}[theorem]{Definition}

\newtheorem*{open question}{Open Question}

\theoremstyle{remark}
\newtheorem{remark}[theorem]{Remark}

\newcommand{\Tor}{\mbox{\rm Tor}}

\newcommand{\Id}{\mathrm{Id}}

\def\Hom{{\rm Hom}}

\def\Ext{{\rm Ext}}

\def\Ker{{\rm Ker}}

\def\qgld{{\rm qgld}}
\def\Cone{{\rm Cone}}
\def\H{{\rm H}}
\def\hsup{{\rm hsup}}
\def\qpd{{\rm qpd}}
\def\pd{{\rm pd}}
\def\EssIm{{\rm EssIm}}

\def\Mod{{\rm Mod}}
\def\mod{{\rm mod}}
\def\rad{{\rm rad}}
\def\Proj{{\rm Proj}}
\def\hinf{{\rm hinf}}
\def\Im{{\rm Im}}
\begin{document}
	
	\begin{center}
		{\large \bf Quasi-projective dimensions over homological ring epimorphisms}
		\vspace{0.5cm}
		
		Weiheng Xia$^a$, Xiaolei Zhang$^b$
		\vspace{0.3cm}
		
		{\footnotesize
			a. School of Mathematical Sciences, Capital Normal University, 100048 Beijing, P. R. China\\
			b. School of Mathematics and Statistics, Tianshui Normal University, 741001 Tianshui, P. R. China
		}
	\end{center}
	
	\bigskip
	\centerline{\bf Abstract}
	\bigskip
	\leftskip10truemm \rightskip10truemm \noindent
	In this paper, we study quasi-projective dimensions of modules over rings connected by homological ring epimorphisms. We establish several inequalities relating the quasi-projective dimensions over the source and target rings. As applications, we investigate the quasi-global dimension of triangular matrix rings and provide examples showing that the class of modules with finite quasi-projective dimension does not satisfy the 2-out-of-3 property in general.
	\vspace{0.3cm}
	
	
	\leftskip0truemm \rightskip0truemm
	
	\section{Introduction}
Throughout, all rings are associative with identity, all modules are left modules. Let $R$ be a ring, denote by $R$-\Proj\ the class of all projective $R$-modules, $R$-\Mod\ the class of all $R$-modules, $R$-\mod\ the class of all finitely generated $R$-modules, and $D(R)$ the unbounded  derived category of $R$.
	
	The quasi-projective dimension, introduced by Gheibi, Jorgensen and Takahashi \cite{gheibi2020}, is a homological invariant that refines the classical projective dimension. For a module $M$ over a ring $R$, $\qpd_R(M)$ is defined as the infimum of $\sup P_\bullet - \hsup P_\bullet$ over all bounded complexes of projective modules whose homology is a direct sum of copies of $M$ (see Definition \ref{def:qpd}). This notion captures a broader class of modules than those of finite projective dimension; for instance, over a complete intersection ring every finitely generated module has finite quasi-projective dimension \cite{gheibi2020}. Subsequent work has revealed deep connections with depth formulas for modules \cite{jorge2024}, derived depth formulas  \cite{ferraro2026}.  Furthermore, Ischebeck's formula and grade inequalities for modules of finite quasi-projective dimension were established in \cite{ischebeck2026}. Recently, the quasi-global dimension $\qgld(R)$ (Definition \ref{def:qgld}) was introduced by Chen, Chen and Liu \cite{chen2025}, and it has been shown that there exist finite-dimensional algebras with finite quasi-global dimension but infinite global dimension. Consequently, this will lead to the need for  bounds or explicit formulas of quasi-global dimensions across a wide range of ring-theoretic constructions, such as trivial extensions, triangular matrix rings, and Morita contexts.

Homological ring epimorphisms, due to Geigle and Lenzing \cite{geigle-lenzing}, are ring epimorphisms $\lambda:R\to S$ such that $\Tor_n^R(S,S)=0$ for all $n\ge1$. They induce fully faithful embeddings of derived categories and are central in tilting theory. In this paper, we investigate how quasi-projective dimensions behave along such epimorphisms. In fact, we show that if $\lambda:R\to S$ is a homological ring epimorphism, then $\qpd_S(M)\le \qpd_R(M)$ for every $S$-module $M$ (Proposition \ref{p7}), and thus
 $\qgld(S)\leq \qgld(R)$ provided  $S$ is a finitely generated $R$-module (Corollary \ref{c8}). We also obtain that if moreover $\pd_R S<\infty$, then $\qpd_S(M)\le \qpd_R(M)\le \qpd_S(M)+\pd_R S$ (Corollary \ref{c11}).

We also study the 2-out-of-3 property for modules of finite quasi-projective dimension over the algebra $A_q=K\langle x,y\rangle/(x^2,y^2,yx+qxy)$ (a local self-injective algebra), and show that this class does not satisfy any of the standard 2-out-of-3 properties (Section 3). Finally, we apply our results to triangular matrix rings, giving bounds on their quasi-global dimensions in terms of the corner algebras (Section 4).

\section{The quasi-projective dimensions on homological ring epimorphisms}

Recall that a ring homomorphism $\lambda:R\to S$ is called a { ring epimorphism} if it is an epimorphism in the category of rings, or equivalently, the canonical map $\lambda_*:S\otimes_R S\to S$ is an isomorphism of $R$-$R$-bimodules.

	\begin{definition}\label{def-s-f}
		An epimorphism $\lambda:R\rightarrow S$ is said to be homological provided that $\Tor_n^R(S,S)=0$ for each $n\geq 1$.
	\end{definition}
	
	\begin{remark}
		Let $\lambda_\ast:S\text{-}\Mod\rightarrow R\text{-}\Mod$ be the restriction functor induced by $\lambda:R\rightarrow S$.
		\begin{enumerate}
			\item We denote by $G:=D(\lambda_\ast):D(S)\rightarrow D(R)$ the derived functor of the exact functor $\lambda_\ast$.
			\item $\lambda$ is a homological ring epimorphism if and only if $G$ is fully faithful if and only if $S\otimes_R^{\mathbb{L}}S\cong S$ in $D(S)$.
		\end{enumerate}
	\end{remark}
	
	Let $G:\mathscr{T}_1\rightarrow \mathscr{T}_2$ be a triangulated functor. Define $\EssIm(G)$ to be the subcategory of $\mathscr{T}_2$ whose objects are $\{X\in \mathscr{T}_2\mid X\cong G(Y)\ \mbox{for some }Y\in\mathscr{T}_1\}$, with morphisms $\Hom_{\EssIm(G)}(X_1,X_2)=\Hom_{\mathscr{T}_2}(X_1,X_2)$.
	
	\begin{lemma}\label{L1}
		Let $G:\mathscr{T}_1\rightarrow \mathscr{T}_2$ be a full triangulated functor. Set $\epsilon:=\EssIm(G)$. Then $\epsilon$ is a full triangulated subcategory of $\mathscr{T}_2$, and is closed under isomorphisms.
	\end{lemma}
	\begin{proof}
		It follows by its definition that $\epsilon$ is a full subcategory of $\mathscr{T}_2$ that is closed under isomorphisms.
		
		First, $\epsilon$ is closed under the shift functor $[1]$ and its inverse $[-1]$: if $X\cong G(Y)$ for some $Y\in\mathscr{T}_1$, then $X[1]\cong G(Y[1])$ and $X[-1]\cong G(Y[-1])$, so both belong to $\EssIm(G)$.
		
		Second, $\epsilon$ is closed under cones of morphisms. Let $f:X_1\rightarrow X_2$ be a morphism in $\epsilon$. Then there exist $Y_1,Y_2\in\mathscr{T}_1$ and isomorphisms $\alpha_i:X_i\cong G(Y_i)$ for $i=1,2$. Consider the morphism $g:=\alpha_2 f \alpha_1^{-1}:G(Y_1)\rightarrow G(Y_2)$ in $\mathscr{T}_2$. Since $G$ is full, there exists $h:Y_1\rightarrow Y_2$ in $\mathscr{T}_1$ such that $G(h)=g$. Now take a cone of $h$, say $Y_3$, fitting into a distinguished triangle $Y_1\xrightarrow{h}Y_2\rightarrow Y_3\rightarrow Y_1[1]$ in $\mathscr{T}_1$. Applying $G$ gives a distinguished triangle in $\mathscr{T}_2$, and by the uniqueness of cones up to isomorphism, we have $G(Y_3)\cong \Cone(f)$. Thus $\Cone(f)\in\EssIm(G)$.
	\end{proof}
	
	Let $X_\bullet$ be a complex. Define $\tau_{\leq n}X_\bullet$ and $\tau_{\geq n}X_\bullet$ to be the good truncations of $X_\bullet$, and $X_{\leq n}$ and $X_{\geq n}$ to be the brutal truncations  of $X_\bullet$.
	
\begin{lemma}\cite[Lemma 13.30.3]{Dsp}\label{dsp}
		Let $F:\mathscr{A}_1\rightarrow \mathscr{A}_2$ and $G:\mathscr{A}_2\rightarrow \mathscr{A}_1$ be adjoint functors between abelian categories. If the derived functors $RF$ and $RG$ exist, then $RF:D(\mathscr{A}_1)\rightarrow D(\mathscr{A}_2)$ and $RG:D(\mathscr{A}_2)\rightarrow D(\mathscr{A}_1)$ are adjoint functors between derived categories.
	\end{lemma}

	\begin{corollary}
		Let $\lambda:R\rightarrow S$ be a homomorphism of rings. Set $G:=D(\lambda_\ast):D(S)\rightarrow D(R)$ and $F:=S\otimes_R^{\mathbb{L}}-:D(R)\rightarrow D(S)$. Then $(F,G)$ is an adjoint pair between $D(R)$ and $D(S)$.
	\end{corollary}
\begin{proof}
	It directly follows by  Lemma \ref{dsp}.	
\end{proof}
	
	\begin{lemma}\label{l4}
		Let $\lambda:R\rightarrow S$ be a homomorphism of rings. Set $G:=D(\lambda_\ast):D(S)\rightarrow D(R)$ and $F:=S\otimes_R^{\mathbb{L}}-:D(R)\rightarrow D(S)$.	Suppose $X_\bullet \in D(R)$ has only finitely many non-zero homology groups, and there exist $N_i\in S\text{-}\Mod$ such that $\H_i(X_\bullet)\cong G(N_i)$ as $R$-modules for each $i\in\mathbb{Z}$. Then $X_\bullet\in \EssIm(G)$.
	\end{lemma}
	\begin{proof}
		By Lemma \ref{L1}, $\EssIm(G)$ is a full triangulated subcategory of $D(R)$ closed under isomorphisms. We prove the result by induction on the number $\sharp(X_\bullet)$ of non-zero homology groups of $X_\bullet$.
		
		If $\sharp(X_\bullet)=0$, then $X_\bullet\cong 0=G(0)$ in $D(R)$, so $X_\bullet\in\EssIm(G)$.
		
		If $\sharp(X_\bullet)=1$, suppose $X_\bullet$ has only non-zero homology at degree $b$. Then $X_\bullet\cong \H_b(X_\bullet)[b]$ in $D(R)$. By assumption, $\H_b(X_\bullet)\cong G(N_b)$, so $X_\bullet\cong G(N_b[b])$. Hence $X_\bullet\in\EssIm(G)$.
		
		Assume the result holds for $\sharp(X_\bullet)=k$. Let $\sharp(X_\bullet)=k+1$. Let $b=\max\{i\mid \H_i(X_\bullet)\neq 0\}$. Consider the distinguished triangle in $D(R)$:
		\[
		\tau_{\leq b-1}X_\bullet\rightarrow X_\bullet\rightarrow \H_b(X_\bullet)[b]\rightarrow \tau_{\leq b-1}X_\bullet[1].
		\]
		By the induction hypothesis, $\tau_{\leq b-1}X_\bullet\in\EssIm(G)$, and $\H_b(X_\bullet)[b]\in\EssIm(G)$ by the previous case. Since $\EssIm(G)$ is closed under cones by Lemma \ref{L1}, it follows that $X_\bullet\in\EssIm(G)$.
	\end{proof}
	
	\begin{lemma}\label{l5}
		Let $\lambda:R\rightarrow S$ be a homological ring epimorphism, and let $X_\bullet\in\EssIm(G)$. Then the unit map $\eta_{X_\bullet}:X_\bullet\rightarrow GF(X_\bullet)$ is an isomorphism in $D(R)$.
	\end{lemma}
	\begin{proof}
		Let $Y_\bullet\in D(S)$. Since $G$ is fully faithful, the counit $\epsilon_{Y_\bullet}:F G(Y_\bullet)\rightarrow Y_\bullet$ is an isomorphism. Now let $X_\bullet\in\EssIm(G)$. Then there exists $Y_\bullet\in D(S)$ such that $X_\bullet\cong G(Y_\bullet)$. The triangle identity for the adjunction gives $\eta_{G(Y_\bullet)} G(\epsilon_{Y_\bullet})=\Id_{G(Y_\bullet)}$. Since $\epsilon_{Y_\bullet}$ is an isomorphism, $G(\epsilon_{Y_\bullet})$ is also an isomorphism, so $\eta_{G(Y_\bullet)}$ is an isomorphism. Thus $\eta_{X_\bullet}$ is an isomorphism.
	\end{proof}
	
	\begin{corollary}\label{c6}
		Let $\lambda:R\rightarrow S$ be a homological ring epimorphism and $P_\bullet\in C^b(R\text{-}\Proj)$. Suppose there exist $N_i\in S\text{-}\Mod$ such that $\H_i(P_\bullet)\cong G(N_i)$ in $R\text{-}\Mod$ for each $i\in\mathbb{Z}$. Then $\H_i(S\otimes_R P_\bullet)\cong N_i$ in $S\text{-}\Mod$. Moreover, if $\H_i(P_\bullet)\in S\text{-}\Mod$, then $\H_i(S\otimes_R P_\bullet)\cong \H_i(P_\bullet)$ in $S\text{-}\Mod$.
	\end{corollary}
	\begin{proof}
		Since each $P_i$ is projective, $F(P_\bullet)=S\otimes_R^{\mathbb{L}}P_\bullet=S\otimes_R P_\bullet$. By Lemma \ref{l4}, $P_\bullet\in\EssIm(G)$. By Lemma \ref{l5}, the unit morphism $P_\bullet\rightarrow GF(P_\bullet)$ is an isomorphism in $D(R)$. Consequently,
		\[
		\H_i(P_\bullet)\cong \H_i(GF(P_\bullet))\cong G(\H_i(F(P_\bullet))).
		\]
		By assumption, $\H_i(P_\bullet)\cong \lambda_\ast(N_i)$, which implies that $G(\H_i(F(P_\bullet)))\cong G(N_i)$ as $R$-modules. Since $G:S\text{-}\Mod\rightarrow R\text{-}\Mod$ is fully faithful, we obtain $H_i(F(P_\bullet))\cong N_i$ as $S$-modules. Thus $\H_i(S\otimes_R P_\bullet)\cong N_i$ as $S$-modules.
	\end{proof}

 Let $X = (\cdots \xrightarrow{\partial_{i+2}} X_{i+1} \xrightarrow{\partial_{i+1}} X_i \xrightarrow{\partial_i} X_{i-1} \xrightarrow{\partial_{i-1}} \cdots)$ be a complex of  $R$-modules. Denote by
	\[
	\begin{cases}
		\sup X = \sup\{i \in \mathbb{Z} \mid X_i
		\not= 0\}, \\
		\inf X = \inf\{i \in \mathbb{Z} \mid X_i
		\not= 0\},
	\end{cases}
	\qquad
	\begin{cases}
		\hsup X = \sup\{i \in \mathbb{Z} \mid \mathrm{H}_i(X)
		\not= 0\}, \\
		\hinf X = \inf\{i \in \mathbb{Z} \mid \mathrm{H}_i(X)
	\not= 0\}.
	\end{cases}
	\]
		Next, we recall the notions of the quasi-projective dimension of a module and the quasi-global dimension of a ring.

\begin{definition}\label{def:qpd}\cite{gheibi2020}
	Let $R$ be a ring and $M$ an $R$-module. A bounded below complex $P_\bullet$ of projective $R$-modules is called a \emph{ quasi-projective resolution} of $M$ if $\H_i(P_\bullet)\cong M^{\oplus a_i}$ for some  non-negative   integers $a_i$, not all zero. The {\em quasi-projective dimension} of $M$ is
	\[
	\qpd_R(M)=\inf\{\sup P_\bullet - \hsup P_\bullet \mid P_\bullet \text{ is a quasi-projective resolution of } M\}.
	\]
	If no such resolution exists, set $\qpd_R(M)=\infty$. So we always have $\qpd_R(0)=\infty$.
\end{definition}

\begin{definition}\label{def:qgld}\cite{chen2025}
	The {\em quasi-global dimension} of $R$ is
	\[
	\qgld(R)=\sup\{\qpd_R(M)\mid M\in R\text{-}\mod\}.
	\]
\end{definition}
	
	\begin{proposition}\label{p7}
		Let $\lambda:R\rightarrow S$ be a homological ring epimorphism and $M$ an $S$-module. Then $\qpd_S(M)\leq \qpd_R(M)$.
	\end{proposition}
	\begin{proof}
		We may assume that $\qpd_R(M)<\infty$. Let $P_\bullet\in C^b(R\text{-}\Proj)$ be a finite quasi-projective resolution of $M$ as an $R$-module. Then $\H_i(P_\bullet)\cong M^{\oplus a_i}\in S\text{-}\Mod$ for some non-negative integers $a_i$, not all zero. It follows by Corollary \ref{c6} that $\H_i(S\otimes_R P_\bullet)\cong M^{\oplus a_i}\in S\text{-}\Mod$. Since $S\otimes_R P_\bullet\in C^b(S\text{-}\Proj)$, this complex is a finite quasi-projective resolution of $M$ as an $S$-module. Therefore
		\[
		\sup(S\otimes_R P_\bullet)\leq \sup(P_\bullet).
		\]
		Since $\H_i(S\otimes_R P_\bullet)\cong M^{\oplus a_i}$, the non-zero homology groups of $S\otimes_R P_\bullet$ are the same as those of $P_\bullet$. Hence $\hsup(S\otimes_R P_\bullet)=\hsup(P_\bullet)$. Consequently,
		\[
		\sup(S\otimes_R P_\bullet)-\hsup(S\otimes_R P_\bullet)\leq \sup(P_\bullet)-\hsup(P_\bullet).
		\]
		Taking the infimum over all such resolutions gives $\qpd_S(M)\leq \qpd_R(M)$.
	\end{proof}
	
	\begin{corollary}\label{c8}
		Let $\lambda:R\rightarrow S$ be a homological ring epimorphism such that $S$ is a finitely generated $R$-module. Then $\qgld(S)\leq \qgld(R)$.
	\end{corollary}
	\begin{proof}
		This follows immediately from Proposition \ref{p7} by taking the supremum over all finitely generated $S$-modules.
	\end{proof}
	
	\begin{lemma}\label{l9}
		Let $R$ be a ring and $Q_\bullet\in C^b(R\text{-}\Mod)$ satisfying $\pd_R(Q_i)\leq d<\infty$ for some $d\geq 0$. Then there exists $P_\bullet\in C^b(R\text{-}\Proj)$ such that $P_\bullet$ is quasi-isomorphic to $Q_\bullet$ and $\sup P_\bullet\leq \sup Q_\bullet+d$.
	\end{lemma}
	\begin{proof}
		We construct $P_\bullet$ by induction on the number $\star(Q_\bullet)$ of non-zero terms of $Q_\bullet$.
		
		If $\star(Q_\bullet)=1$, suppose $Q_s\neq 0$. Then there is a projective resolution $P_\bullet$ of $Q_s$ of length at most $d$:
	$${\xymatrix{
			P_\bullet:&	0\ar[r]^{}\ar[d]&P_d\ar[r]^{}\ar[d]&\cdots\ar[r]^{}
			&P_1\ar[d]\ar[r]^{}&P_0\ar@{->>}[d]\ar[r]^{}&0\\
			&	0\ar[r]^{}&0\ar[r]^{}&\cdots\ar[r]^{}
			& 0\ar[r]^{}
			&Q_s\ar[r]^{}&0.}}$$
		This gives a complex $P_\bullet$ with $\sup P_\bullet\leq s+d$.
		
		Assume the result holds for $\star(Q_\bullet)=k$. Let $\star(Q_\bullet)=k+1$, and let $s=\sup Q_\bullet$. Consider the distinguished triangle in $D(R)$:
		\[
		Q_{\leq s-1}\rightarrow Q_\bullet\rightarrow Q_s[s]\xrightarrow{\delta} Q_{\leq s-1}[1].
		\]
		By the induction hypothesis, there is a quasi-isomorphism $f:P'_\bullet\cong Q_{\leq s-1}$ such that $\sup P'_\bullet\leq (s-1)+d$ and $P'_\bullet\in C^b(R\text{-}\Proj)$. By the case $\star(Q_\bullet)=1$, there is a quasi-isomorphism $P^s_\bullet\cong Q_s[s]$ such that $\sup P^s_\bullet\leq s+d$ and $P^s_\bullet\in C^b(R\text{-}\Proj)$.
		
		Set $\widetilde{\delta}:=g\delta(f[1])^{-1}:P^s_\bullet\rightarrow Q_s[s]\rightarrow Q_{\leq s-1}[1]\rightarrow P'_\bullet[1]$. Since $P^s_\bullet\in C^b(R\text{-}\Proj)$, by standard results
on the homotopy category, there exists a morphism $\widetilde{d}\in \Hom_{K(R)}(P^s_\bullet,P'_\bullet[1])$ whose image in $D(R)$ is $\widetilde{\delta}$. Consider the split short exact sequence
		\[
		0\rightarrow P'_\bullet\rightarrow \Cone(\widetilde{d}[-1])\rightarrow P^s_\bullet\rightarrow 0
		\]
		of chain complexes. Then $\Cone(\widetilde{d}[-1])\in C^b(R\text{-}\Proj)$ and $\sup \Cone(\widetilde{d}[-1])\leq s+d$. Moreover, there is a quasi-isomorphism $\Cone(\widetilde{d}[-1])\rightarrow Q_\bullet$. Thus $P_\bullet:=\Cone(\widetilde{d}[-1])$ is the required projective complex.
	\end{proof}
	
	\begin{proposition}\label{p10}
		Let $F:S\text{-}\Mod\rightarrow R\text{-}\Mod$ be an exact functor such that for every projective $S$-module $Q$, we have $\pd_R(F(Q))\leq d<\infty$. Then
		\[
		\qpd_R(F(M))\leq \qpd_S(M)+d
		\]
		for any $S$-module $M$.
	\end{proposition}
	\begin{proof}
		We may assume $\qpd_S(M)=n<\infty$. If $F(M)=0$, then $\qpd_R(F(M))=-\infty$, and the result holds.
		
		Now assume $F(M)\neq 0$. Take a finite quasi-projective resolution $P_\bullet\in C^b(S\text{-}\Proj)$ of $M$. Then $\H_i(P_\bullet)\cong M^{\oplus a_i}$ for some non-negative integers $a_i$, not all zero. Set $s=\sup P_\bullet$ and $h=\hsup P_\bullet$. Then $n=s-h$. Since $F$ is exact, $\H_i(F(P_\bullet))\cong F(M)^{\oplus a_i}$. So $\hsup F(P_\bullet)=h$. Since $\pd_R(F(P_i))\leq d$, by Lemma \ref{l9} there exists $Q_\bullet\in C^b(R\text{-}\Proj)$ quasi-isomorphic to $F(P_\bullet)$ such that $\sup Q_\bullet\leq \sup F(P_\bullet)+d$ and $\hsup Q_\bullet=h$. Since $\sup F(P_\bullet)\leq \sup P_\bullet$, we have $\sup Q_\bullet\leq \sup P_\bullet+d=s+d$. Thus $Q_\bullet$ is a finite quasi-projective resolution of $F(M)$, so
		\[
		\qpd_R(F(M))\leq \sup Q_\bullet-\hsup Q_\bullet\leq s+d-h=n+d.
		\]
		This proves the result.
	\end{proof}
	
	\begin{corollary}\label{c11}
		Let $\lambda:R\rightarrow S$ be a ring homomorphism with $\pd_R S=d<\infty$. Then
		\[
		\qpd_R(M)\leq \qpd_S(M)+d
		\]
		for any $S$-module $M$.
	\end{corollary}
	\begin{proof}
		Apply Proposition \ref{p10} to the restriction functor $\lambda_\star:S\text{-}\Mod\rightarrow R\text{-}\Mod$, which is exact and sends projectives to modules of projective dimension at most $d$.
	\end{proof}
	
	\begin{corollary}
		Let $\lambda:R\rightarrow S$ be a homological ring epimorphism with $\pd_R S=d<\infty$. Then
		\[
		\qpd_S(M)\leq \qpd_R(M)\leq \qpd_S(M)+d
		\]
		for any $S$-module $M$.
	\end{corollary}
	\begin{proof}
		This follows directly from Proposition \ref{p7} and Corollary \ref{c11}.
	\end{proof}
	
	\section{2-out-of-3 properties for modules with finite quasi-projective dimensions}
	
	In this section, we investigate the class of modules with finite quasi-projective dimension. In particular, we show that this class does not satisfy any of the standard 2-out-of-3 properties for short exact sequences, in contrast to the class of modules of finite projective dimension.
	
	Throughout this section, let $K$ be a field, $q\in K^{\star}$ not a root of unity, and define
	\[
	A_q = K\langle x,y\rangle/(x^2,y^2,yx+qxy).
	\]
	Then $A_q$ is a local self-injective algebra with basis $\{1,x,y,xy\}$ over $K$, and its radical is $J:=\rad(A_q)=(x,y,xy)$.
	
	\begin{lemma}\label{l13}
		Let $S$ be the unique simple module of $A_q$. Then $\qpd_{A_q}(S)=0$.
	\end{lemma}
	\begin{proof}
		Consider the complex
		\[
		P_\bullet:\cdots\rightarrow 0\rightarrow A_q\xrightarrow{d_2} A_q^2\xrightarrow{d_1} A_q\rightarrow 0,
		\]
		where $d_2:A_q\rightarrow A_q^2$ is given by $c\mapsto (cy,qcx)$, and $d_1:A_q^2\rightarrow A_q$ is given by $(a,b)\mapsto ax+by$. These are $A_q$-homomorphisms and $d_2d_1=0$. We show that $P_\bullet$ is a quasi-projective resolution of $S$.
		
		For $H_0(P_\bullet)$, since $\Im(d_1)=A_qx+A_qy=J$, we have $\H_0(P_\bullet)=A_q/J\cong S$.
		
		For $H_1(P_\bullet)$, we compute $\Ker(d_1)$ and $\Im(d_2)$. Let $a=\alpha_0+\alpha_1x+\alpha_2y+\alpha_3xy$ and $b=\beta_0+\beta_1x+\beta_2y+\beta_3xy$. The condition $d_1(a,b)=0$ gives $\alpha_0=\beta_0=\beta_1-q\alpha_2=0$. Thus $\Ker(d_1)$ has basis $\{(x,0),(0,y),(y,qx),(xy,0),(0,xy)\}$. Meanwhile, $\Im(d_2)$ has basis $\{(y,qx),(xy,0),(0,xy)\}$. Hence $H_1(P_\bullet)$ has basis $\{\overline{(x,0)},\overline{(0,y)}\}$, and $J H_1(P_\bullet)=0$, so $H_1(P_\bullet)\cong S^2$.
		
		For $H_2(P_\bullet)$, if $d_2(c)=0$ for $c=\delta_0+\delta_1x+\delta_2y+\delta_3xy$, then $\delta_0=\delta_1=\delta_2=0$. Hence $\Ker(d_2)=Kxy\cong S$, so $H_2(P_\bullet)\cong S$.
		
		Thus $P_\bullet$ is a quasi-projective resolution of $S$ with $\sup P_\bullet=\hsup P_\bullet=2$, so $\qpd_{A_q}(S)=0$.
	\end{proof}
	
	\begin{lemma}\label{l14} \cite[Section 4.1]{Erdmann2017}
		Let $\lambda\in K^{\star}$. Define $C(\lambda):=A_q(x+\lambda y)=\langle xy,x+\lambda y\rangle_K$. Then:
		\begin{enumerate}
			\item $\dim_K(C(\lambda))=2$, and hence $C(\lambda)\notin A_q\text{-}\Proj$.
			\item $\Omega^m(C(\lambda))\cong C(q^{-m}\lambda)$.
			\item $\Ext_{A_q}^{\geq 2}(C(\lambda),C(\lambda))=0$ and $\Ext_{A_q}^{1}(C(\lambda),C(\lambda))\neq 0$.
		\end{enumerate}
	\end{lemma}

	\begin{lemma}\label{l15} \cite[Theorem 1.2, Corollary 2.4]{chen2025}
		Let $\mathscr{A}$ be an abelian category with enough projectives.
		\begin{enumerate}
			\item Let $M\in\mathscr{A}$ with $\qpd_{\mathscr{A}}(M)<\infty$ and $\Ext_{\mathscr{A}}^{\geq 2}(M,M)=0$. Then $\pd_{\mathscr{A}}(M)<\infty$.
			\item Let $0\rightarrow N\rightarrow E\rightarrow M\rightarrow 0$ be an exact sequence in $\mathscr{A}$, and suppose $E$ is projective-injective in $\mathscr{A}$. Then $\qpd_{\mathscr{A}}(M)\leq \qpd_{\mathscr{A}}(N)+1$.
			\item Moreover, if $\mathscr{A}$ is a Frobenius abelian category, then $\qpd_{\mathscr{A}}(M)=0$ or $\infty$ for any $M\in\mathscr{A}$.
		\end{enumerate}
	\end{lemma}

	\begin{lemma}\label{l16}\cite[Proposition 3.3, Proposition 3.6]{gheibi2020}
		Let $\mathscr{A}$ be an abelian category with enough projectives. Let $M,N\in\mathscr{A}$. Then:
		\begin{enumerate}
			\item $\qpd_{\mathscr{A}}(M\oplus N)\leq \sup\{\qpd_{\mathscr{A}}(M),\qpd_{\mathscr{A}}(N)\}$.
			\item $\qpd_{\mathscr{A}}(\Omega(M))\leq \qpd_{\mathscr{A}}(M)$.
			\item If there is a positive integer $n$ such that $\Omega^n(M)\cong M$, then $\qpd_{\mathscr{A}}(M)=0$.
		\end{enumerate}
	\end{lemma}

	Let $\mathscr{A}$ be an abelian category with enough projectives. Define $\mathscr{Q}:=\{M\in\mathscr{A}\mid \qpd_{\mathscr{A}}(M)<\infty\}$. The following example shows that $\mathscr{Q}$ does not satisfy any 2-out-of-3 property.
	
	\begin{example}
		Let $A_q=K\langle x,y\rangle/(x^2,y^2,yx+qxy)$, and let $C(\lambda):=A_q(x+\lambda y)=\langle xy,x+\lambda y\rangle_K$ as above. Then:
		\begin{enumerate}
			\item In the natural exact sequence $0\rightarrow K(xy)\rightarrow C(\lambda)\rightarrow C(\lambda)/K(xy)\rightarrow 0$, we have $\qpd_{A_q}(K(xy))=\qpd_{A_q}(C(\lambda)/K(xy))=0$, but $\qpd_{A_q}(C(\lambda))=\infty$.
			\item In the natural exact sequence $0\rightarrow C(\lambda)\rightarrow J\rightarrow J/C(\lambda)\rightarrow 0$, we have $\qpd_{A_q}(J)=\qpd_{A_q}(J/C(\lambda))=0$, but $\qpd_{A_q}(C(\lambda))=\infty$.
			\item In the natural exact sequence $0\rightarrow C(q^{-1}\lambda)/K(xy)\rightarrow A_q/K(xy)\rightarrow C(\lambda)\rightarrow 0$, we have $\qpd_{A_q}(C(q^{-1}\lambda)/K(xy))=\qpd_{A_q}(A_q/K(xy))=0$, but $\qpd_{A_q}(C(\lambda))=\infty$.
		\end{enumerate}
	\end{example}
	\begin{proof}
		For (1), since $J(K(xy))=0$ and $\dim_K(K(xy))=1$, we have $K(xy)\cong S$. Similarly, $C(\lambda)/K(xy)\cong S$. Thus $\qpd_{A_q}(K(xy))=\qpd_{A_q}(C(\lambda)/K(xy))=0$ by Lemma \ref{l13}. Since $A_q$ is self-injective, by Lemma \ref{l15}(3), $\qpd_{A_q}(C(\lambda))$ is either $0$ or $\infty$. If it were $0$, then by Lemma \ref{l15}(1) and Lemma \ref{l14}(3), $C(\lambda)$ would have finite projective dimension, hence be projective by self-injectivity, contradicting Lemma \ref{l14}(1). Thus $\qpd_{A_q}(C(\lambda))=\infty$.
		
		For (2), note that $0\rightarrow J\rightarrow A_q\rightarrow S\rightarrow 0$ is exact. By Lemma \ref{l16}(2), $\qpd_{A_q}(J)\leq \qpd_{A_q}(S)=0$, so $\qpd_{A_q}(J)=0$. Also $J^2=\langle xy\rangle_K\subseteq C(\lambda)$, so $J(J/C(\lambda))=0$ and $\dim_K(J/C(\lambda))=1$, hence $J/C(\lambda)\cong S$. Therefore $\qpd_{A_q}(J/C(\lambda))=0$.
		
		For (3), consider the exact sequence $0\rightarrow C(q^{-1}\lambda)\xrightarrow{i} A_q\xrightarrow{\pi} C(\lambda)\rightarrow 0$, with $K(xy)\subseteq C(q^{-1}\lambda)\subseteq A_q$. By the snake lemma applied to the diagram
		\[
		\begin{array}{ccccccccc}
			0&\rightarrow& K(xy)&\rightarrow& C(q^{-1}\lambda)&\rightarrow& C(q^{-1}\lambda)/K(xy)&\rightarrow&0\\
			&&\downarrow&&\downarrow&&\downarrow&&\\
			0&\rightarrow& K(xy)&\rightarrow& A_q&\rightarrow& A_q/K(xy)&\rightarrow&0
		\end{array}
		\]
		we obtain the exact sequence $0\rightarrow C(q^{-1}\lambda)/K(xy)\rightarrow A_q/K(xy)\rightarrow C(\lambda)\rightarrow 0$. Since $J(C(q^{-1}\lambda)/K(xy))=0$ and $\dim_K(C(q^{-1}\lambda)/K(xy))=1$, we have $C(q^{-1}\lambda)/K(xy)\cong S$. Also $K(xy)\cong S$, so by Lemma \ref{l15}(2), $\qpd_{A_q}(A_q/K(xy))\leq 1$. By Lemma \ref{l15}(3), it follows that $\qpd_{A_q}(A_q/K(xy))=0$.
	\end{proof}
	
	This example shows that the class of modules of finite quasi-projective dimension behaves quite differently from the class of modules of finite projective dimension, which does satisfy the 2-out-of-3 property.
	
	\section{Quasi-global dimension of triangular matrix rings}
	
	Let $B$ and $C$ be Artin algebras, and let $U$ be a finitely generated $C$-$B$-bimodule. Define the triangular matrix ring
	\[
	A:=\begin{pmatrix}
		B & 0 \\
		_C U_B & C
	\end{pmatrix}.
	\]
	Set
	\[
	e_B:=\begin{pmatrix}
		1_B & 0 \\
		0 & 0
	\end{pmatrix}, \qquad
	e_C:=\begin{pmatrix}
		0 & 0 \\
		0 & 1_C
	\end{pmatrix}.
	\]
	
	\begin{remark}
		\begin{enumerate}
			\item Any left $A$-module corresponds to a triple $(X,Y,\phi)$, where $X$ is a $B$-module, $Y$ is a $C$-module, and $\phi:U\otimes_B X\rightarrow Y$ is a $C$-homomorphism.
			\item As abelian groups, $(X,Y,\phi)\cong X\oplus Y$, and the action is given by
			\[
			\begin{pmatrix}
				b & 0 \\
				u & c
			\end{pmatrix}(x,y)=(bx,\phi(u\otimes x)+cy).
			\]
			\item There are four natural functors:
			\begin{itemize}
				\item $j_!:B\text{-}\mod\rightarrow A\text{-}\mod$ given by $X\mapsto (X,U\otimes_B X,\Id)$.
				\item $i_\star:C\text{-}\mod\rightarrow A\text{-}\mod$ given by $Y\mapsto (0,Y,0)$.
				\item $F_B:A\text{-}\mod\rightarrow B\text{-}\mod$ given by $(X,Y,\phi)\mapsto X$.
				\item $F_C:A\text{-}\mod\rightarrow C\text{-}\mod$ given by $(X,Y,\phi)\mapsto Y$.
			\end{itemize}
			\item The functors $F_B$ and $F_C$ are exact. Moreover, $F_B$ sends projective $A$-modules to projective $B$-modules. If $\pd_C U\leq d$, then $\pd_C(F_C(P))\leq d$ for each projective $A$-module $P$. The functor $i_\star$ is exact and sends projectives to projectives. The functor $j_!$ sends projectives to projectives, and if $U\in B^{op}\text{-}\Proj$, then $j_!$ is exact.
		\end{enumerate}
	\end{remark}
	
	We now recall the notion of stratifying ideals and their relation to homological ring epimorphisms.
	
	\begin{definition}\cite{AKSLY17}
		Let $R$ be a ring and $e$ an idempotent in $R$. The ideal $ReR$ is called a stratifying ideal if the multiplication map
		\[
		Re\otimes_{eRe} eR\rightarrow ReR
		\]
		is an isomorphism as $R$-$R$-bimodules, and $\Tor_n^{eRe}(Re,eR)=0$ for all $n\geq 1$.
	\end{definition}
	
	\begin{remark}
		If $ReR$ is a stratifying ideal, then the natural surjection $R\twoheadrightarrow R/ReR$ is a homological ring epimorphism.
	\end{remark}
	The following result is easy to verify and so left as an excise.
	\begin{lemma}
		Let $A=\begin{pmatrix} B & 0 \\ U & C \end{pmatrix}$ and let $e=e_C=\begin{pmatrix} 0 & 0 \\ 0 & 1_C \end{pmatrix}$. Then:
		\begin{enumerate}
			\item $AeA$ is a stratifying ideal of $A$.
			\item $eAe\cong C$ and $A/AeA\cong B$ as algebras.
		\end{enumerate}
	\end{lemma}
	
	\begin{corollary}\label{c20}
		$\qgld(B)\leq \qgld(A)$. Moreover, if $\pd_C(U)=d<\infty$, then $\qgld(C)\leq \qgld(A)+d$.
	\end{corollary}
	\begin{proof}
		The first inequality follows from Proposition \ref{p7} applied to the homological epimorphism $A\rightarrow A/AeA\cong B$. The second follows from Proposition \ref{p10} applied to the restriction functor $A\text{-}\Mod\rightarrow C\text{-}\Mod$.
	\end{proof}

Following by  the proof of \cite[Proposition 3.2]{SLM25}, we can obtain the following result.	
	\begin{lemma}\label{l21}
		Let $\mathscr{A}$ be an abelian category with enough projectives, and let $0\rightarrow P\rightarrow M\rightarrow N\rightarrow 0$ be an exact sequence with $P\in\mathscr{A}\text{-}\Proj$. Then
		\[
		\qpd_{\mathscr{A}}(N)\leq \max\{1,\qpd_{\mathscr{A}}(M)\}.
		\]
	\end{lemma}

	\begin{proposition}\label{p22}
		Let $C$ be a semisimple algebra and assume $U\in B^{op}\text{-}\Proj$. Then
		\[
		\qgld(B)\leq \qgld(A)\leq \max\{1,\qgld(B)\}.
		\]
		Consequently, if $\qgld(B)\geq 1$, then $\qgld(A)=\qgld(B)$.
	\end{proposition}
	\begin{proof}
		The left inequality follows from Corollary \ref{c20}. For the right inequality, take $M=(X,Y,\phi)\in A\text{-}\mod$. Since $C$ is semisimple, we may write $Y=\Im(\phi)\oplus W$ for some $W$. Then $M\cong (X,\Im(\phi),\phi)\oplus (0,W,0)$. Let $M'=(X,\Im(\phi),\phi)$.  We may assume $M'\not=0$.  Since $C$ is semisimple, $W$ is projective as a $C$-module, so $(0,W,0)$ is projective as an $A$-module. By Lemma \ref{l16}(1), $\qpd_A(M)\leq \qpd_A(M')$.
		
		Now consider the exact sequence
		\[
		0\rightarrow i_\star(\Ker(\phi))\rightarrow j_!(X)\rightarrow M'\rightarrow 0.
		\]
		Since $\Ker(\phi)$ is projective as a $C$-module (as a submodule of a projective module over a semisimple algebra), $i_\star(\Ker(\phi))$ is projective as an $A$-module. By Lemma \ref{l21}, $\qpd_A(M')\leq \max\{1,\qpd_A(j_!(X))\}$. Since $U\in B^{op}\text{-}\Proj$, the functor $j_!$ is exact, so by Proposition \ref{p10}, $\qpd_A(j_!(X))\leq \qpd_B(X)$. Therefore $\qpd_A(M)\leq \max\{1,\qpd_B(X)\}\leq \max\{1,\qgld(B)\}$. Taking the supremum over all $M$ gives the result.
	\end{proof}

	The following example shows that the condition ``$C$ is semisimple'' cannot be omitted in Proposition \ref{p22}.
	
	\begin{example}
		Let $C=K[x]/\langle x^2\rangle$ where $K$ is a field. Set $B=K$ and $U=C/\langle x\rangle$. Then $\qgld(B)=0$ and $\qgld(A)=\infty$.
	\end{example}
	\begin{proof}
		Trivially, $\qgld(B)=\qgld(K)=0$. Every indecomposable projective  $A$-module is of the form $P_1=(K,U\otimes_K K,\Id)\cong (K,U,0)$ or $P_2=(0,C,0)$. Let $M=(K,0,0)\in A\text{-}\mod$. We show $\qpd_A(M)=\infty$.
		
		First, note that $\pd_A(M)=\infty$. Indeed, consider the exact sequence
		\[
		0\rightarrow (0,U,0)\rightarrow (K,U,0)\xrightarrow{(\Id_K,0)} M\rightarrow 0.
		\]
		Setting $N:=(0,U,0)$, there is an exact sequence $0\rightarrow N\rightarrow P_2\rightarrow N\rightarrow 0$, so $N$ has infinite projective dimension, hence so does $M$.
		
		Now we claim $\Ext_A^n(M,M)=0$ for all $n\geq 1$. From the exact sequence $0\rightarrow N\rightarrow P_1\rightarrow M\rightarrow 0$, we get $\Ext_A^1(M,M)=0$ since $\Hom_A(N,M)=0$. Then $\Ext_A^n(M,M)=0$ for $n\geq 2$ by dimension shifting. By Lemma \ref{l15}(1), $\qpd_A(M)=\infty$.
	\end{proof}
	
	The following example shows that the condition ``$U\in B^{op}\text{-}\Proj$'' cannot be omitted in Proposition \ref{p22}.
	
	\begin{example}
		Let $B=K[x]/\langle x^2\rangle$ where $K$ is a field. Set $C=K$ and $U=B/\langle x\rangle$. Then $\qgld(B)=0$ and $\qgld(A)=\infty$.
	\end{example}
	\begin{proof}
		Since the indecomposable $R$-modules are only of the form $B$ and $U$ up to isomorphism. Since $\Omega_B(U)\cong U$, it follows by Lemma \ref{l16}(3) that $\qgld(B)=0.$
		
		Every indecomposable $A$-module is of the forms $P_1=(B,U\otimes_BB,\Id_{U\otimes_BB})\cong (B,U,\theta)$ where $\theta:U\otimes_BB\rightarrow U$ is the natural isomorphism or $P_2=(0,K,0).$
		Set $M=(U,U,\mu)$ where $\mu:U\otimes_BU\rightarrow U$ is the multiplication map. We will show $\qpd_AM=\infty$. First we will show $\pd_AM=\infty.$ Take $p:=(\pi,\Id_U):(B,U,\theta)\rightarrow M$ where $\pi$ is the natural epimorphism. Then $p$ is an epimorphism of $A$-modules. Consider the exact sequence $$0\rightarrow ((x)/(x^2),0,0)\rightarrow (B,U,\theta)\xrightarrow{p}M\rightarrow0.$$ Note that $N:=(U,0,0)\cong ((x)/(x^2),0,0)$ as $A$-modules. There exists  a short exact sequence $$0\rightarrow ((x)/(x^2),0,0)\rightarrow (B,U,\theta)\xrightarrow{q:=(\pi,0)}N\rightarrow0$$  $A$-modules. So $\Ker(q)\cong (U,U,0)\cong (U,0,0)\oplus(0,U,0)=N\oplus(0,U,0)$. Since $\pd_AN=\infty$, $\pd_AM=\infty$.
		
		Next, we will show $\Ext_A^n(M,M)=0$ for each $n\geq 1.$ We claim that $\Hom_A(N,M)=0$. For any $h\in \Hom_A(N,M)$
		Then $h=(f,0)$ where $f:U\rightarrow U$ is a $B$-homomorphism. Since $h$ is an $A$-homomorphism, we have the following commutative diagram:
		$${\xymatrix{
				U\otimes_BU\ar[r]^{0}\ar[d]^{\Id_U\otimes f}&0\ar[d]^0\\
				U\otimes_BU\ar[r]^{\mu}&U}}$$
		Since $U$ is a simple $B$-module, $f$ must be isomorphism or $0$. It follows by the above commutative diagram that $f=0$, and so $h=0,$ and hence $\Hom_A(N,M)=0$.
		
		Considering the following short exact sequence $0\rightarrow N\rightarrow P_1\rightarrow M\rightarrow0$, we have a long exact sequence $\Hom_A(P_1,M)\rightarrow \Hom_A(N,M)\rightarrow\Ext_A^1(M,M)\rightarrow0$. Since $\Hom_A(N,M)=0$, $\Ext_A^1(M,M)=0.$
		
		Set $P=(0,U,0).$ Then there exists a short exact $0\rightarrow N\oplus P\rightarrow P_1\rightarrow N\rightarrow0$. And so there exists a short exact sequence $\Hom_A(P_1,M)\xrightarrow{(i,\Id_M)}\Hom_A(N\oplus P)\rightarrow\Ext_A^1(N,M)\rightarrow0.$ For any $\alpha=(0,g):P=(0,U,0)\rightarrow (U,U,\mu)=M$. Set $g(\overline{1})=\lambda \overline{1}$ for some $\lambda\in K$, define $l:B\rightarrow U$ satisfying $l(b)=\lambda\pi(b)$. Then $l$ is a $B$-homomorphism. Note that $(l,g):P_1=(B,U,\theta)\rightarrow(U,U,\mu)=M$ is an $A$-homomorphism, and $(0,g)=(l,g)|_{P}$. Consequently, $(i,\Id_M)$ is surjective. It follows by dimension shift that $\Ext_A^n(M,M)=0$ for any $n\geq 2.$ It follows by Lemma \ref{l15}(1) that $\qpd_A(M)=\infty.$	
	\end{proof}

	\bibliographystyle{amsplain}

\end{document}